\documentclass[12pt,a4paper]{article}
\usepackage[latin5]{inputenc}
\usepackage[T1]{fontenc}
\usepackage{amsmath}
\numberwithin{equation}{section}
\usepackage{amsfonts}
\usepackage{amssymb}
\usepackage{amsthm}
\usepackage{graphicx}
\usepackage{enumerate}
\usepackage{enumitem}
\usepackage{multicol}
\usepackage[all]{xy}
\usepackage{xcolor}
\usepackage{soul}
\usepackage{faktor}
\usepackage{bm}
\usepackage{easyReview}

\usepackage{geometry}
\allowdisplaybreaks
\usepackage{indentfirst}
\usepackage{tocloft}
\usepackage{authblk}
\usepackage{lmodern}
\usepackage{sansmath}

\sffamily\sansmath

\theoremstyle{plain}
\newtheorem{theorem}{Theorem}[section]
\newtheorem{proposition}{Proposition}[section]
\newtheorem{lemma}{Lemma}[section]
\newtheorem{corollary}{Corollary}[section]

\theoremstyle{definition}
\newtheorem{definition}{Definition}[section]
\newtheorem{example}{Example}[section]
\newtheorem{remark}{Remark}[section]

\def\Gwa{$\mathsf{{Gr}^{\bullet}}$}
\def\im{\operatorname{Im}}
\def\Aut{\operatorname{Aut}}
\def\St{\operatorname{St}}
\def\cok{\operatorname{coker}}

\def\wt{\widetilde}

\begin{document}
\title{\bf Coverings and liftings of generalized crossed modules}

\author[ ]{Gamze AYTEKİN ARICI}
\author[ ]{Tunçar ŞAHAN\thanks{\textbf{Corresponding author: }\texttt{tuncarsahan@gmail.com} (T. Şahan)}}
\affil[ ]{\small{Department of Mathematics, Aksaray University, Aksaray, Turkey}}

\date{}

\maketitle

\begin{abstract}
In the theory of crossed modules, considering arbitrary self-actions instead of conjugation allows for the extension of the concept of crossed modules and thus the notion of generalized crossed module emerges. In this paper we give a precise definition for generalized cat$^1$-groups and obtain a functor from the category of generalized cat$^1$-groups to generalized crossed modules. Further, we introduce the notions of coverings and liftings for generalized crossed modules and investigate properties of these structures. Main objective of this study is to obtain an equivalence between the category of coverings and the category of liftings of a given generalized crossed module $(A,B,\alpha)$.
\end{abstract}


\section{Introduction}
Crossed modules are defined by Whitehead \cite{Whitehead1946} as an  algebraic model for homotopy 2-types and have been widely used in different areas of mathematics (see for example  \cite{Brown2011,Brown1982a,Huebschmann1980,Loday1978,Lue1979}). Crossed modules can be seen as 2-dimensional groups \cite{Brown1982a}. It is a well-known fact that the category of crossed modules over groups and the category of group-groupoids, which also known as 2-groups, are naturally equivalent \cite{Brown1976,Loday1982}. Using this equivalence many properties and structures known in one of these categories are interpreted in other one. Some examples for this interpretations are normality and quotient objects \cite{Mucuk2015}, coverings \cite{Akiz2013} and actions \cite{Mucuk2019}. These studies also give rise to similar studies on higher dimensional crossed modules \cite{Demir2021,Temel2020,Temel2020a} and on topological crossed modules \cite{Mucuk2020,Mucuk2014}.

In 1987, Porter \cite{Porter1987} gave a general definition of crossed modules in certain algebraic categories namely categories of groups with operations. See also \cite{Orzech1972,Orzech1972a} for an extensive research on categories of groups with operations and on categories of interest. As a special case $(\mathcal{C}=\mathbf{Gp})$ of Porter's definition for crossed modules, a group homomorphism $\alpha\colon A\rightarrow B$ with a derived action of $B$ on $A$ is a crossed module if the following diagram of split extensions is commutative in the category of groups
\[\xymatrix{
\bm{\mathsf{0}} \ar[r] &   A \ar@{->}[r]^-{\iota}\ar@{->}[d]_{1_{A}}  &   A\rtimes A \ar@{->}[r]_-{p} \ar@{->}[d]_{1_{A}\times\alpha} & A \ar@/_/[l]_-{s} \ar[r] \ar@{->}[d]^{\alpha} & \bm{\mathsf{0}} \\
\bm{\mathsf{0}} \ar[r] & A \ar@{->}[r]^-{\iota} \ar@{->}[d]_{\alpha} &   A\rtimes B \ar@{->}[r]_-{p} \ar@{->}[d]_{\alpha\times 1_{B}} & B \ar@{->}[d]^{1_{B}} \ar@/_/[l]_-{s} \ar[r] & \bm{\mathsf{0}}\\
\bm{\mathsf{0}} \ar[r] & B \ar@{->}[r]^-{\iota}  & B\rtimes B \ar@{->}[r]_-{p}  &  B \ar@/_/[l]_-{s} \ar[r]  & \bm{\mathsf{0}} }\]
where the top and bottom rows are constructed from conjugation actions on $A$ and on $B$, respectively, and the middle row constructed from the derived action of $B$ on $A$. Referring to the conventional definition of crossed modules, conjugation actions are used on both groups. Recently, Yavari and Salemkar \cite{Yavari2019} examined the effects of taking arbitrary actions on both groups instead of conjugation actions in the definition of crossed modules. They named these structures by generalized crossed modules since every crossed module is a generalized one. 

In this study, with a similar thought we define generalized cat$^1$-groups and a functor from the category of generalized cat$^1$-groups to that of generalized crossed modules. Moreover, we introduce the notions of coverings and liftings for generalized crossed modules and investigate the properties of these kind of structures. Finally, we prove that the category of covering generalized crossed modules and the category of lifting generalized crossed modules are naturally equivalent for a fixed generalized crossed module $(A,B,\alpha)$.

\section{Preliminaries}

A \textit{crossed $B$-module of groups} (or briefly a \textit{crossed module}) is a group homomorphism $\alpha\colon A\rightarrow B$ with a (left) group action of $B$ on $A$, denoted by $b\cdot a$, satisfying
\begin{enumerate}[label={(\roman{*})}, leftmargin=1cm]
	\item $\alpha(b\cdot a)=b+\alpha(a)-b$,
	\item $\alpha(a)\cdot a_1=a+a_1-a$
\end{enumerate}
for all $a\in A$ and $b\in B$. 

Such a crossed module is denoted by $(A,B,\alpha)$.

\begin{example}\label{ex:xmod}
	Following are the very well-known examples for crossed modules.
\begin{enumerate}[label=(\roman{*}), leftmargin=1cm]
	\item Any group $G$ and its any normal subgroup $N$ defines a crossed module with the embedding $N\hookrightarrow G$ with the conjugation action of $G$ on $N$. As a special case of this example $(G,G,1_G)$ becomes a crossed module.
	
	\item Let $M$ be a $G$-module. Then  $0\colon M\rightarrow G$, the zero map, has a structure of crossed module.
	
	\item Automorphism group $\Aut(G)$ of a given group $G$ acts on $G$ in a very natural way, say $f\cdot g=f(g)$ for all $f\in\Aut(G)$ and $g\in G$. In this case the inner automorphism map $G\rightarrow Aut(G)$ becomes a crossed module.
	
	\item\label{ex:xmodtop} If $X$ is a topological group, then the fundamental group $\pi X$ is a group-groupoid, the star  $\St_{\pi X}0$ at the identity $0\in X$ becomes a group and the final point map $d_1\colon\St_{\pi X}0 \rightarrow X$ becomes a crossed module.
	
	\item The origin of crossed modules is based on this example due to Whitehead \cite{Whitehead1946,Whitehead1949} if $X$ is  topological space and $A\subseteq X$ with $x\in A$, then  there is a natural action of $\pi_1(A,x)$ on second relative homotopy group $\pi_2(X,A,x)$ and with this action the boundary map \[\partial\colon\pi_2(X,A,x)\rightarrow \pi_1(A,x)\] becomes a crossed module. This crossed module is called {\em fundamental crossed module} and denoted by $\Pi(X,A,x)$ (see for example \cite{Brown2011} for more details).
\end{enumerate}
\end{example}

Let $(A,B,\alpha)$ and $(A',B',\alpha')$ be two crossed modules and let $f\colon A\rightarrow A'$ and $g\colon B\rightarrow B'$ be two group homomorphisms. If $g\alpha=\alpha' f$ and $f(b\cdot a)=g(b)\cdot f(a)$ for all $a\in A$ and $B\in B$ then $\langle f,g\rangle $ is called a morphism of crossed modules. The conditions for crossed module morphisms mean that the diagram below is commutative.
\[\xymatrix{
	B\times A \ar[d]_{g\times f} \ar[r]^-{\cdot} & A \ar[d]_{f} \ar[r]^-{\alpha}  & B \ar[d]^-{g}  \\
	B'\times A'  \ar[r]_-{\cdot} & A' \ar[r]_-{\alpha'}  & B' }
\]
Crossed modules and crossed module morphisms given above forms a category. We will write $\mathsf{XMod}$ for the category of crossed modules.

It is a well-known fact that the category of crossed modules and of group-groupoids, which are categories internal to groups, are naturally equivalent \cite{Brown1976}. 

Now we recall the definition of a generalized crossed module from \cite{Yavari2019}. First we need to remind the notion of group with action due to Datuashvili \cite{Datuashvili2002a,Datuashvili2004}.

\begin{definition}\cite{Datuashvili2002a}
	A  group with action is a group $G$ with an (left) action on itself.
\end{definition} 

The action of $G$ on itself will be denoted by
\[ \begin{array}{ccl}
	G\times G & \longrightarrow & G                 \\
	\left(g_1,g_2\right) & \longmapsto     & {}^{g_{1}}{g_{2}}
\end{array}\] 
for all $g_1,g_2\in G$.

A morphism $f$ between groups with action $G$ and $H$ is a group homomorphism $f\colon G\rightarrow H$ such that $f({}^{g}{g_1})={}^{f(g)}{f(g_1)}$ for all $g,g_1\in G$, i.e. $f$ preserves the self-action. The category of groups with action and morphisms between them is denoted by \Gwa{} \cite{Datuashvili2002a}.

\begin{definition}\cite{Datuashvili2002a}
A subgroup $H$ of a group with action $G$ is called a subobject of $G$ if $H$ is closed under the action of $G$ on itself.
\end{definition}

\begin{definition}\cite{Datuashvili2002a}\label{def:ideal}
Let $N$ be a subobject of a group with action $G$. Then $N$ is called an ideal of $G$ if 
\begin{enumerate}[label=(\roman{*}), leftmargin=1cm]
\item\label{cond:ideal1} $N$ is a normal subgroup of $G$,
\item\label{cond:ideal2} ${}^{g}{n}\in N$ and
\item\label{cond:ideal3} ${}^{n}{g}-g\in N$
\end{enumerate}
for all $g\in G$ and $n\in N$.
\end{definition}

It has been shown in \cite{Datuashvili2002a} that the quotient group $G/N$ is an object in \Gwa{} with the self-action given by
\[{}^{(g+N)}{(g_1+N)}=({}^{g}{g_1})+N\]
for all $g,g_1\in G$.

\subsection{Generalized crossed modules}

\begin{definition}\label{def:genxmod}
Let $A$ and $B$ be two objects in \Gwa{}, let $\alpha\colon A\rightarrow B$ be a group homomorphism and let there is a group action of $B$ on $A$ which is denoted by $b\cdot a$ for $a\in A$ and $b\in B$. Then $(A,B,\alpha)$ is called a \textbf{generalized crossed module}, if
 \begin{enumerate}[label=(\roman{*}), leftmargin=1cm]
 	\item\label{defcond:genxmod1} $\alpha(b\cdot a)={}^{b}{\alpha(a)}$,
 	\item\label{defcond:genxmod2} $\alpha(a)\cdot a_1={}^{a}{a_1}$
 \end{enumerate}
for all $a,a_1\in A$ and $b\in B$ \cite{Yavari2019}.
\end{definition}

\begin{remark}
It is easy to see that for a generalized crossed module $(A,B,\alpha)$
\begin{alignat*}{2}
	\alpha({}^{a}{a_1})&=\alpha(\alpha(a)\cdot a_1) \tag{by (ii)}\\
	&={}^{\alpha(a)}{\alpha(a_1)} \tag{by (i)}
\end{alignat*}
for all $a,a_1\in A$, that is, the group homomorphism $\alpha$ is in fact a morphism in \Gwa{}.
\end{remark}

A generalized crossed module $(A,B,\alpha)$ is called \textbf{aspherical} if $\ker \alpha=0$ and \textbf{simply connected} if $\cok \alpha=0$. In other words, $(A,B,\alpha)$ is an aspherical generalized crossed module if $\alpha$ is injective, and a simply connected generalized crossed module if $\alpha$ is surjective. 

\begin{example}
	If the self-actions on $A$ and on $B$  are conjugation action then a generalized crossed module $(A,B,\alpha)$ becames a usual crossed module. So every crossed module is a generalized crossed module.
\end{example} 

\begin{example}
	Let $A$ and $B$ be two groups. Assume that all actions, i.e. of $A$ on itself, of $B$ on itself and on $A$, are trivial. Then for any group homomorphism $\alpha\colon A\rightarrow B$ the triple $(A,B,\alpha)$ is a generalized crossed module.
\end{example}

\begin{example}\label{ex:subgrpgxmod}
Let $G$ be an object in \Gwa{} and let $H$ be a subgroup of $G$ such that ${}^{g}{h}\in H$ for all $g\in G$ and $h\in H$. Then $(H,G,i)$ becomes an aspherical generalized crossed module where the action of $G$ on $H$ is the induced one from the self-action of $G$ and $i\colon H\rightarrow G$ is the inclusion map. In particular for any characteristic subgroup $K$ of $G$ the triple $(K,G,i)$ is a generalized crossed module. It can be seen that if $N$ is an ideal of $G$, as defined in Definition \ref{def:ideal}, then $(N,G,i)$ is a generalized crossed module.
\end{example}

This example is a generalization of normal subgroup crossed module where the action is conjugation.

\begin{lemma}
Let $(A,B,\alpha)$ be a generalized crossed module. Then $(\ker \alpha,A,i)$ becomes a generalized crossed module as in Example \ref{ex:subgrpgxmod}
\end{lemma}

\begin{proof}
Since $\ker \alpha$ is a subgroup of $A$ then it is sufficient to prove that ${}^{a}{x}\in \ker\alpha$ for all $a\in A$ and $x\in \ker\alpha$. $x\in \ker\alpha$ implies that $\alpha(x)=0$. Thus $\alpha({}^{a}{x})={}^{\alpha(a)}{\alpha(x)}={}^{\alpha(a)}{0}=0$ and hence ${}^{a}{x}\in \ker\alpha$.
\end{proof}

\begin{lemma}
Let $(A,B,\alpha)$ be a generalized crossed module and let $K:=\alpha(A)$. Then $(K,B,i)$ becomes a generalized crossed module as in Example \ref{ex:subgrpgxmod}
\end{lemma}

\begin{proof}
It is well known that $K$ is a subgroup of $B$. We only need to show that ${}^{b}{k}\in K$ for all $b\in B$ and $k\in K$. There is an element $a$ in $A$ such that $\alpha(a)=k$ since $K:=\alpha(A)$. Then ${}^{b}{k}={}^{b}{\alpha(a)}=\alpha(b\cdot a)\in \alpha(A)=K$. This completes the proof.
\end{proof}

\begin{lemma}
Let $(A,B,\alpha)$ be a generalized crossed module. Then $\ker \alpha$ acts trivially on $A$.
\end{lemma}

\begin{proof}
Let $x\in \ker \alpha$, i.e. $\alpha(x)=0$ and let $a\in A$. Then ${}^{x}{a}=\alpha(x)\cdot a= 0\cdot a=a$. This completes the proof.
\end{proof}

Let $(A,B,\alpha)$ and $(A',B',\alpha')$ be two generalized crossed modules and let $f\colon A\rightarrow A'$ and $g\colon B\rightarrow B'$ be two group homomorphisms. If $g\alpha=\alpha' f$ and $f(b\cdot a)=g(b)\cdot f(a)$ for all $a\in A$ and $B\in B$ then $\langle f,g\rangle $ is called a morphism of generalized crossed modules. All generalized crossed modules and morphisms between them as defined above forms a category. We will write $\mathsf{\bf GXMod}$ for the category of generalized crossed modules.

\begin{lemma}\label{lem:isogxmod1}
For any generalized crossed module $(A,B,\alpha)$ let $f\colon B\rightarrow B'$ be an isomorphism in \Gwa{}. Then $(A,B',\widetilde{\alpha})$ is a generalized crossed module where $\widetilde{\alpha}=f\alpha$ and the action of $B'$ on $A$ is given by \[b'\cdot a=\left( f^{-1}(b')\right)\cdot a\] for $b'\in B'$ and $a\in A$.  
\end{lemma}

\begin{proof}
We need to show that $(A,B',\widetilde{\alpha})$ satisfies the conditions \ref{defcond:genxmod1} and \ref{defcond:genxmod2} of Definition \ref{def:genxmod}.
\begin{enumerate}[label=(\roman{*}), leftmargin=1cm]
\item Let $a\in A$ and $b'\in B'$. Then
\begin{alignat*}{2}
	\widetilde{\alpha}(b'\cdot a) & = \widetilde{\alpha}(f^{-1}(b')\cdot a) \\
	& = f\left( {\alpha}(f^{-1}(b')\cdot a) \right) \\
	& = f\left({}^{f^{-1}(b')}{\alpha(a)}\right) \\
	& = {}^{ff^{-1}(b')}{f\alpha(a)} \\
	& = {}^{b'}{f\alpha(a)} \\
	& = {}^{b'}{\widetilde{\alpha}(a)}.
\end{alignat*}
\item Let $a,a_1\in A$. Then
\begin{alignat*}{2}
	\widetilde{\alpha}(a)\cdot a_1 & = f^{-1}(\widetilde{\alpha}(a))\cdot a_1 \\
	& = f^{-1}(f{\alpha}(a))\cdot a_1 \\
	& = \alpha(a)\cdot a_1 \\
	& = {}^{a}{a_1}.
\end{alignat*}
\end{enumerate}
Thus $(A,B',\widetilde{\alpha})$ is a generalized crossed module.
\end{proof}

In this case $(A,B',\widetilde{\alpha})$ is isomorphic to $(A,B,\alpha)$.

\begin{lemma}\label{lem:isogxmod2}
Let $(A,B,\alpha)$ be a generalized crossed module and let $g\colon A'\rightarrow A$ be an isomorphism in \Gwa{}. Then $(A',B,\widehat{\alpha})$ is a generalized crossed module where $\widehat{\alpha}=\alpha g$ and the action of $B$ on $A'$ is given by \[b\cdot a'=g^{-1}\left(b\cdot g(a')\right)\] for $b\in B$ and $a'\in A'$.  
\end{lemma}

\begin{proof}
We need to show that $(A',B,\widehat{\alpha})$ satisfies the conditions \ref{defcond:genxmod1} and \ref{defcond:genxmod2} of Definition \ref{def:genxmod}.
\begin{enumerate}[label=(\roman{*}), leftmargin=1cm]
\item Let $a'\in A'$ and $b\in B$. Then
\begin{alignat*}{2}
\widehat{\alpha}(b\cdot a') & = \widehat{\alpha}(g^{-1}\left(b\cdot g(a')\right)) \\
& = {\alpha g}(g^{-1}\left(b\cdot g(a')\right)) \\
& = \alpha\left(b\cdot g(a')\right) \\
& = {}^{b}{\alpha g(a')} \\
& = {}^{b}{\widehat{\alpha}(a')}
\end{alignat*}
\item Let $a',a'_1\in A'$. Then
\begin{alignat*}{2}
\widehat{\alpha}(a')\cdot a'_1 & = g^{-1}\left(\widehat{\alpha}(a')\cdot g(a'_1)\right) \\
& = g^{-1}\left(\alpha(g(a'))\cdot g(a'_1)\right) \\
& = g^{-1}\left({}^{g(a')}{g(a'_1)}\right) \\
& = {}^{g^{-1}g(a')}{g^{-1}g(a'_1)} \\ 
& = {}^{a'}{a'_1}.
\end{alignat*}
\end{enumerate}
Thus $(A',B,\widehat{\alpha})$ is a generalized crossed module.
\end{proof}

In this case $(A',B,\widehat{\alpha})$ is isomorphic to $(A,B,\alpha)$.

\begin{corollary}
Let $(A,B,\alpha)$ be a generalized crossed module and let $f\colon B\rightarrow B'$, $g\colon A'\rightarrow A$ be two isomorphisms in \Gwa{}. Then $(A',B',\gamma)$ is a generalized crossed module where $\gamma=f\alpha g$ and the action of $B'$ on $A'$ is given by \[b'\cdot a'=g^{-1}\left(f^{-1}(b')\cdot g(a')\right)\] for $b'\in B'$ and $a'\in A'$. Moreover $(A,B',\widetilde{\alpha})$, $(A',B,\widehat{\alpha})$, $(A',B',\gamma)$ and $(A,B,\alpha)$ are all isomorphic generalized crossed modules.
\end{corollary}

\subsection{Generalized cat$^{1}$-groups}

\begin{definition}\label{def:gencat1gp}
Let $G$ be an object in \Gwa{} and $s,t\colon G\rightarrow G$ be two morphisms in \Gwa{}. Then $(G,s,t)$ is called a \textbf{generalized cat$^1$-group} if
\begin{enumerate}[label={(\roman{*})}, leftmargin=1cm]
	\item\label{defcond:gencat1gp1} $st=t$, $ts=s$ and
	\item\label{defcond:gencat1gp2} ${}^{y}{x}=x$ for all $x\in\ker s$ and $y\in\ker t$.
\end{enumerate}
\end{definition}

\begin{remark}
Here note that for a generalized cat$^1$-group $(G,s,t)$ if the action of $G$ on itself is conjugation then ${}^{y}{x}=y+x-y=x$ and hence $y+x=x+y$ for all $x\in\ker s$ and $y\in\ker t$. This means that $(G,s,t)$ is an ordinary cat$^1$-group. Thus every cat$^1$-group is a generalized cat$^1$-group where the self action is conjugation.
\end{remark}

\begin{example}
Let $G$ be group with trivial action on itself. Then $(G,1_G,1_G)$ is a generalized cat$^1$-group.
\end{example}

Let $(G,s,t)$ and $(G',s',t')$ be two generalized cat$^1$-groups and let $f\colon G\rightarrow G'$ a morphism in \Gwa{}. If $fs=s'f$ and $ft=t'f$ then $f$ is called a morphism of generalized cat$^1$-groups.

\[\xymatrix{
	 G \ar[d]_{f} \ar@<0.5ex>[r]^-{s} \ar@<-0.5ex>[r]_-{t} & G \ar[d]^-{f}  \\
	 G' \ar@<0.5ex>[r]^-{s'} \ar@<-0.5ex>[r]_-{t'}  & G' }
\]

The category of all generalized cat$^1$-groups and morphisms between them is denoted by $\mathsf{\bf GC^{1}Gp}$.

\begin{proposition}
For a generalized cat$^1$-group $(G,s,t)$ the triple $(\ker s,\im s, \overline{t})$ is a generalized crossed module where $\overline{t}=t|_{\ker s}$.
\end{proposition}

\begin{proof}
Here the action of $\im s$ on $\ker s$ is given by $x\cdot g={}^{x}{g}$ for all $g\in\ker s$ and $x\in\im s$. It is clear that $\ker s$ and $\im s$ are objects of \Gwa{} with the induced action from that of $G$ since $s$ is a morphism in \Gwa{}, and that $\overline{t}$ is a group homomorphism. So we only need to show that conditions \ref{defcond:genxmod1} and \ref{defcond:genxmod2} of Definition \ref{def:genxmod} are satisfied.  
\begin{enumerate}[label=(\roman{*}), leftmargin=1cm]
	\item Let $g\in\ker s$ and $x\in\im s$. Then there exist an element $g_1\in G$ such that $s(g_1)=x$ and
	\begin{alignat*}{2}
		\overline{t}(x\cdot g)&=\overline{t}({}^{x}{g}) \\
		& = t({}^{x}{g}) \\
		& = {}^{t(x)}{t(g)} \\
		& = {}^{t(s(g_1))}{t(g)} \\
		& = {}^{s(g_1)}{t(g)} \\
		& = {}^{x}{t(g)}.
	\end{alignat*}
	\item Let $g,g'\in \ker s$. Then $-g+t(g)\in \ker t$ and by the condition \ref{defcond:gencat1gp2} of Definition \ref{def:gencat1gp} ${}^{-g+t(g)}{g'}=g'$. Hence
	\begin{alignat*}{2}
		{\overline{t}(g)}\cdot{g'}&={}^{\overline{t}(g)}{g'} \\
		& = {}^{t(g)}{g'} \\
		& = {}^{g+(-g+t(g))}{g'} \\
		& = {}^{g}\left({{}^{(-g+t(g))}{g'}}\right) \\
		& = {}^{g}{g'}.
	\end{alignat*}
\end{enumerate}
This completes the proof.
\end{proof}

\begin{remark}
Above construction defines a functor 
\begin{equation*}
\mathsf{\bf GC^{1}Gp}\longrightarrow\mathsf{\bf GXMod}
\end{equation*}
from the category of generalized cat$^1$-groups to that of generalized crossed modules.
\end{remark}

\section{Coverings of generalized crossed modules}

The notion of covering of a crossed module is introduced in \cite{Brown1994} by Brown and Mucuk. They proved that the category of covering groups of a given topological group $X$ and the category of covering crossed modules of the associated crossed module to $X$, as in Example \ref{ex:xmod} \ref{ex:xmodtop}, are naturally equivalent. In this section we introduce the notion of coverings of generalized crossed modules and explore their properties. 

\begin{definition}
Let $(A,B,\alpha)$ and $(\wt{A},\wt{B},\wt{\alpha})$ be two generalized crossed modules. If there is a morphism $\langle f,g\rangle \colon (\wt{A},\wt{B},\wt{\alpha})\rightarrow (A,B,\alpha)$ of generalized crossed modules such that $f\colon \wt{A}\rightarrow A$ is an isomorphism then $(\wt{A},\wt{B},\wt{\alpha})$ is called a covering of $(A,B,\alpha)$ and $\langle f,g\rangle$ is called a covering morphism.
\end{definition}

\[\xymatrix{\wt{A} \ar[d]_{f}^{\cong} \ar[r]^-{\wt{\alpha}}  & \wt{B} \ar[d]^-{g}  \\ A \ar[r]_-{\alpha}  & B }\]

\begin{example}
Let $(A,B,\alpha)$ be a generalized crossed module. Then $\langle 1_A,1_B \rangle\colon (A,B,\alpha)\rightarrow (A,B,\alpha)$ is a covering morphism of generalized crossed modules.
\end{example}

Let $\langle \wt{f},\wt{g}\rangle \colon (\wt{A},\wt{B},\wt{\alpha})\rightarrow (A,B,\alpha)$ and $\langle f',g'\rangle \colon (A',B',\alpha')\rightarrow (A,B,\alpha)$ be two covering morphisms. If there is a morphism $\langle f,g\rangle \colon (\wt{A},\wt{B},\wt{\alpha})\rightarrow (A',B',\alpha')$ of generalized crossed modules such that $\langle f',g'\rangle\circ\langle f,g\rangle=\langle \wt{f},\wt{g}\rangle$ then $\langle f,g\rangle$ is called a morphism of coverings of $(A,B,\alpha)$.

\[\xymatrix@R=1.5cm@C=0.5cm{
	(\wt{A},\wt{B},\wt{\alpha}) \ar[rr]^{\langle f,g\rangle} \ar[dr]_{\langle \wt{f},\wt{g}\rangle} &  &  (A',B',\alpha') \ar[dl]^{\langle f',g'\rangle} \\
	& (A,B,\alpha) &  }
\]

Coverings of a generalized crossed module $(A,B,\alpha)$ and morphisms between them as defined above forms a category. This category is denoted by $\mathsf{\bf Cov_{GXMod}}(A,B,\alpha)$.

\begin{proposition}\label{prop:kernel}
Let  $\langle f,g\rangle \colon (\wt{A},\wt{B},\wt{\alpha})\rightarrow (A,B,\alpha)$ be a covering morphism of generalized crossed modules. Then $f(\ker \wt{\alpha}) \subseteq \ker \alpha$.
\end{proposition}

\begin{proof}
Let $\wt{a}\in\ker \wt{\alpha}$. Then $\wt{\alpha}(\wt{a})=0$. Then we get 
\begin{equation*}
\alpha(f(\wt{a})) = g(\wt{\alpha}(\wt{a})) = g(0) = 0
\end{equation*}
since $\langle f,g\rangle$ is a morphism of generalized crossed modules and hence $f(\wt{a})\in \ker \alpha$. This completes the proof.
\end{proof}

Let $(A,B,\alpha)$ be an aspherical generalized crossed module and let $\langle f,g\rangle \colon (\wt{A},\wt{B},\wt{\alpha})\rightarrow (A,B,\alpha)$ be a covering morphism of generalized crossed modules. Since $(A,B,\alpha)$ is aspherical then $\ker\alpha=0$. By Proposition \ref{prop:kernel} $f(\ker \wt{\alpha})=0$ and consequently $\ker \wt{\alpha}=0$, since $f$ is an isomorphism.

\begin{corollary}
Every covering of an aspherical generalized crossed module is also aspherical.
\end{corollary}

\begin{example}
Since $(H,G,i)$ is an aspherical generalized crossed module (Example \ref{ex:subgrpgxmod}) then every covering of $(H,G,i)$ is also aspherical.
\end{example}

\begin{lemma}
Let $\langle f,g\rangle \colon (\wt{A},\wt{B},\wt{\alpha})\rightarrow (A,B,\alpha)$ be a covering morphism of generalized crossed modules, and let $h\colon \wt{B}\rightarrow\wt{C}$ and $k\colon B\rightarrow C$ be two isomorphisms in \Gwa{}. Then \[\langle f,\wt{g}\rangle \colon (\wt{A},\wt{C},h\wt{\alpha})\rightarrow (A,C,k\alpha)\] is also a covering morphism of generalized crossed modules where $\wt{g}=kgh^{-1}$.
\end{lemma}
\begin{proof}
It is sufficient to prove the commutativity of the following diagram:
\[\xymatrix{\wt{A} \ar[d]_{f}^{\cong} \ar[r]^-{h\wt{\alpha}}  & \wt{C} \ar[d]^-{\wt{g}}  \\ A \ar[r]_-{k\alpha}  & C. }\]
Then we have
\begin{equation*}
\wt{g}(h\wt{\alpha})=kgh^{-1}(h\wt{\alpha})=kg\wt{\alpha}=k\alpha f
\end{equation*}
which completes the proof.
\end{proof}

\begin{proposition}
Let $\langle f,g\rangle \colon (\wt{A},\wt{B},\wt{\alpha})\rightarrow (A,B,\alpha)$ and $\langle \wt{f},\wt{g}\rangle \colon (A',B',\alpha')\rightarrow (\wt{A},\wt{B},\wt{\alpha})$ be two covering morphism of generalized crossed modules. Then so is $\langle f',g'\rangle \colon (A',B',\alpha')\rightarrow (A,B,\alpha)$ where $f'=f\mbox{}\wt{f}$ and $g'=g\wt{g}$.
\end{proposition}

\[\xymatrix{A' \ar[d]_{\wt{f}}^{\cong}  \ar@<-1pt> `l[d] `[dd]_{f'} [dd] \ar[r]^-{\alpha'}  & B' \ar[d]^-{\wt{g}} \ar@<-1pt> `r[d] `[dd]^{g'} [dd]  \\ \wt{A} \ar[d]_{f}^{\cong} \ar[r]^-{\wt{\alpha}}  & \wt{B} \ar[d]^-{g}  \\ A \ar[r]_-{\alpha}  & B }\]

\begin{proof}
The proof is easy so is omitted.
\end{proof}

\begin{lemma}
Let $\langle \wt{f},\wt{g}\rangle \colon (\wt{A},\wt{B},\wt{\alpha})\rightarrow (A,B,\alpha)$ and $\langle f',g'\rangle \colon (A',B',\alpha')\rightarrow (A,B,\alpha)$ be two covering morphisms, and let $\langle f,g\rangle \colon (\wt{A},\wt{B},\wt{\alpha})\rightarrow (A',B',\alpha')$ be a morphism between these two coverings of $(A,B,\alpha)$. Then $\langle f,g\rangle$ becomes a covering morphism of generalized crossed modules.
\end{lemma}

\begin{proof}
It is easy to see that $f$ is an isomorphism since $f'$ and $\wt{f}$ are isomorphisms and $f=(f')^{-1}\wt{f}$. This completes the proof.
\end{proof}

\begin{theorem}
Let $(C,D,\gamma)$ be a simply connected generalized crossed module, let \[\langle f,g\rangle\colon (C,D,\gamma)\rightarrow(A,B,\alpha)\] be a morphism of generalized crossed modules and let $\langle \wt{f},\wt{g}\rangle\colon (\widetilde{A},\widetilde{B},\widetilde{\alpha})\rightarrow(A,B,\alpha)$ be a covering morphism. Then there exist a generalized crossed module morphism \[\langle f',g'\rangle\colon (C,D,\gamma)\rightarrow(\widetilde{A},\widetilde{B},\widetilde{\alpha})\] such that $\langle \wt{f},\wt{g}\rangle\langle f',g'\rangle=\langle f,g\rangle$ if and only if $f(\ker\gamma)\subseteq \wt{f}(\ker\widetilde{\alpha})$.
\end{theorem}

\begin{proof}
Let assume that $f(\ker\gamma)\subseteq \wt{f}(\ker\widetilde{\alpha})$. We take $f'={\wt{f}}^{-1}f$ and define $g'\colon D\rightarrow \wt{B}$ as $g'(d)=\wt{\alpha}f'(c)$ for any $c\in C$ satisfying $\gamma(c)=d$. Proving that $\langle f',g'\rangle$ is a generalized crossed module morphism is straightforward from assumptions.

On the other side, let $c\in\ker\gamma$. Then we get
\begin{equation*}
0=g'(0)=\wt{\alpha}(f'(c))=\wt{\alpha}({\wt{f}}^{-1}(f(c)))
\end{equation*}
by the definition of $g'$ and hence ${\wt{f}}^{-1}(f(c))\in\ker\wt{\alpha}$ which implies $f(c)\in\wt{f}(\ker\wt{\alpha})$ and completes the proof.
\end{proof}

\begin{corollary}
Let $\langle \wt{f},\wt{g}\rangle\colon (\widetilde{A},\widetilde{B},\widetilde{\alpha})\rightarrow(A,B,\alpha)$ and $\langle f',g'\rangle\colon (A',B',\alpha')\rightarrow(A,B,\alpha)$ be two covering morphisms such that $(A',B',\alpha')$ is simply connected. Then $(A',B',\alpha')$ is also a covering of $(\widetilde{A},\widetilde{B},\widetilde{\alpha})$ if and only if $f(\ker\alpha')\subseteq \wt{f}(\ker\widetilde{\alpha})$.
\end{corollary}

\begin{corollary}
Let $\langle \wt{f},\wt{g}\rangle\colon (\widetilde{A},\widetilde{B},\widetilde{\alpha})\rightarrow(A,B,\alpha)$ and $\langle f',g'\rangle\colon (A',B',\alpha')\rightarrow(A,B,\alpha)$ be two covering morphisms such that $(\widetilde{A},\widetilde{B},\widetilde{\alpha})$ and $(A',B',\alpha')$ are both simply connected. Then there is an isomorphism $\langle f,g\rangle\colon(A',B',\alpha')\rightarrow(\widetilde{A},\widetilde{B},\widetilde{\alpha})$ if and only if $f(\ker\alpha')= \wt{f}(\ker\widetilde{\alpha})$.
\end{corollary}

\subsection{An application to topological groups with action}

\begin{definition}
Let $X$ be a topological group. If there is a continuous group action of $X$ on itself then we say that $X$ is a \textbf{self-acting topological group}.
\end{definition}

We will denote the set of all homotopy classes of the paths in $X$ by $\pi(X)$ and the set of all homotopy classes of the paths with the initial point $x_0$ in $X$ by $\mathcal{P}(X,x_0)$. That is, $\mathcal{P}(X,x_0)=\left\{ [\alpha] \mid \alpha(0)=x_0\right\}$.

\begin{proposition}
Let $X$ be a self-acting topological group. Then $\pi(X)$ is a group with action with the induced action from that of $X$, i.e. 
\[\begin{array}{rcccl}
\phi & : & \pi(X)\times \pi(X)        & \longrightarrow & \pi(X)\\
     &   & \big([\alpha],[\beta]\big) & \longmapsto     & {~}^{[\alpha]}{[\beta]}=\left[ {}^{\alpha}{\beta} \right]
\end{array}\]
where $({}^{\alpha}{\beta})(t)={}^{\alpha(t)}{\beta(t)}$ for $t\in[0,1]$.
\end{proposition}


\begin{proof}
It is a well known fact that $\pi(X)$ is a group with the operation induced from that of $X$. So we only need to show that $\phi$ is a group action. Now we prove that $\phi$ is well-defined. Let denote the self-action of $X$ with $\psi$. Assume that $F:\alpha\simeq \alpha'$ and $G:\beta\simeq \beta'$. Then $H:=\psi\circ(F,G)$ is a homotopy from ${~}^{\alpha}{\beta}$ to ${~}^{\alpha'}{\beta'}$. Other details are straightforward from the definition of ${~}^{\alpha}{\beta}$ and the fact that $X$ being a self-acting topological group.
\end{proof}

\begin{proposition}
Let $X$ be a self-acting topological group with the identity element $e$. Then $(\mathcal{P}(X,e)$ is an ideal of $\pi(X)$. \end{proposition}
\begin{proof}
We already know that $\mathcal{P}(X,e)$ is a normal subgroup of $\pi(X)$. We only need to show that conditions \ref{cond:ideal2} and \ref{cond:ideal3} of Definition \ref{def:ideal} are satisfied. Let $[\alpha]\in\pi(X)$ and $[\beta]\in\mathcal{P}(X,e)$.
\begin{enumerate}[leftmargin=1cm]
\item[(ii)] Then $\beta(0)=e$ and we get
\begin{equation*}
\left({}^{\alpha}{\beta}\right)(0)={}^{\alpha(0)}{\beta(0)}={}^{\alpha(0)}{e}=e
\end{equation*}
and this implies that ${}^{[\alpha]}{[\beta]}\in\mathcal{P}(X,e)$.
\item[(iii)] For the third condition we get
\begin{equation*}
\left({}^{\beta}{\alpha}-\alpha\right)(0)={}^{\beta(0)}{\alpha(0)}-\alpha(0)={}^{e}{\alpha(0)}-\alpha(0)=\alpha(0)-\alpha(0)=e
\end{equation*}
which implies ${}^{[\beta]}{[\alpha]}-[\alpha]\in\mathcal{P}(X,e)$.
\end{enumerate}
\end{proof}

\begin{example}
As in Example \ref{ex:subgrpgxmod} for a self-acting topological group $X$ with the identity element $e$ the triple $(\mathcal{P}(X,e),\pi(X),i)$ is an aspherical generalized crossed module.
\end{example}

\begin{corollary}
Let $\wt{X}$ and $X$ be two self-acting topological groups and $p\colon\wt{X}\rightarrow X$ be a covering morphism of topological groups such that $p\left({}^{{\wt{x}}_{1}}{{\wt{x}}_{2}}\right)={}^{p({\wt{x}}_{1})}{p({\wt{x}}_{2})}$ for any ${\wt{x}}_{1},{\wt{x}}_{2}\in\wt{X}$. Then \[\langle p^{\ast}|_{\mathcal{P}(\wt{X},\wt{e})},p^{\ast}\rangle\colon(\mathcal{P}(\wt{X},\wt{e}),\pi(\wt{X}),\wt{i})\rightarrow(\mathcal{P}(X,e),\pi(X),i)\] becomes a covering morphism of generalized crossed modules where $p^{\ast}\colon\pi(\wt{X})\rightarrow\pi(X)$ given by $p^{\ast}([\wt{\alpha}])=[p\wt{\alpha}]$ for any $[\wt{\alpha}]\in\pi(\wt{X})$. 
\end{corollary}

%
%

\section{Liftings of generalized crossed modules}

Liftings of crossed modules were defined by Mucuk and Şahan \cite{Mucuk2019} as the interpretations of group-groupoid actions on groups in the category of crossed modules. See \cite{Tuncar2019,Demir2021} for further details on liftings. In that paper they prove a natural equivalence between the category of coverings of a given crossed module and the category of liftings of that crossed module. In the light of this result, now we define the notion of a lifting of a generalized crossed module and give a generalization of the equivalence given in \cite{Mucuk2019}.

\begin{definition}
Let $(A,B,\alpha)$ be a generalized crossed module, let $X$ be an object in \Gwa{} and let $\omega\colon X\rightarrow B$ be a group homomorphism. In this case there is an action of $X$ on $A$ via $\omega$, i.e. $x\cdot a=\omega(x)\cdot a$. If there is a generalized crossed module $(A,X,\varphi)$ such that $\omega\varphi=\alpha$ then we say that $(A,X,\varphi)$ is a lifting crossed module of $(A,B,\alpha)$.
\end{definition}
\[\xymatrix{ & X \ar[d]^\omega\\
	A \ar[r]_-\alpha \ar@{-->}[ur]^\varphi & B}\]

\begin{example}
Let $(A,B,\alpha)$ be a generalized crossed module. Then $(A,\alpha(A),\alpha)$ is a lifting of $(A,B,\alpha)$ over the inclusion homomorphism $i\colon \alpha(A)\rightarrow B$.
\end{example}

\begin{example}
Let $(A,B,\alpha)$ be a generalized crossed module and let $(A,B',\widetilde{\alpha})$ be the generalized crossed module as in Lemma \ref{lem:isogxmod1}. Then $(A,B,\alpha)$ is a lifting of $(A,B',\widetilde{\alpha})$ over $f\colon B\rightarrow B'$.
\end{example}

\begin{example}
Let $(A,B,\alpha)$ be a generalized crossed module and let $(A',B,\widehat{\alpha})$ be the generalized crossed module as in Lemma \ref{lem:isogxmod2}. Since every isomorphism in \Gwa{} is a generalized crossed module then $(A',A,g)$ is a lifting of $(A',B,\widehat{\alpha})$ over $\alpha\colon A\rightarrow B$.
\end{example}

\begin{example}
Let $(A,B,\alpha)$ be a generalized crossed module. Then $(A,A/\ker\alpha,p)$ becomes a lifting of $(A,B,\alpha)$ over $\omega\colon A/\ker\alpha\rightarrow B$, $a+\ker\alpha\mapsto\alpha(a)$. In this example $(A,A/\ker\alpha,p)$ is called the natural lifting of $(A,B,\alpha)$.
\end{example}

\begin{lemma}\label{lem:liftprexmod}
Let $(A,B,\alpha)$ be a generalized crossed module, let $X$ be an object in \Gwa{} and let $\varphi\colon A\rightarrow X$, $\omega\colon X\rightarrow B$ be two group homomorphisms such that $\omega\varphi=\alpha$. Then $(A,X,\varphi)$ is a generalized crossed module, hence a lifting of $(A,B,\alpha)$, if and only if $\varphi(x\cdot a)={}^{x}{\varphi(a)}$ for all $x\in X$ and $a\in A$.
\end{lemma}

\begin{proof}
We only need to show that $(A,X,\varphi)$ satisfies condition \ref{defcond:genxmod2} of Definition \ref{def:genxmod} since condition \ref{defcond:genxmod1} is already satisfied from the assumption. Let $a,a_1\in A$. Then
\begin{alignat*}{2}
{\varphi(a)}\cdot{a_1} & = \omega(\varphi(a))\cdot{a_1} \\
& = \alpha(a)\cdot a_1 \\
& = {}^{a}{a_1}.
\end{alignat*}
This completes the proof.
\end{proof}

\begin{proposition}
Let $(A,B,\alpha)$ be a generalized crossed module and let $(A,X,\varphi)$ be a lifting of $(A,B,\alpha)$ over the group homomorphism $\omega\colon X\rightarrow B$. Then $\ker \varphi \subseteq \ker \alpha$ and $\langle 1_A, \omega \rangle\colon (A,X,\varphi)\rightarrow(A,B,\alpha)$ is a generalized crossed module morphism.
\end{proposition}

\begin{proof}
Let $x\in \ker \varphi$. Then $\varphi(x)=0$ and
\begin{alignat*}{2}
\alpha(x) & = \omega\varphi(x) \\
& = \omega(\varphi(x)) \\
& = \omega(0) \\
& = 0.
\end{alignat*} 
Thus $x\in \ker\alpha$ and $\ker \varphi \subseteq \ker \alpha$. In order to prove that $\langle 1_A, \omega \rangle\colon (A,X,\varphi)\rightarrow(A,B,\alpha)$ is a generalized crossed module morphism, first we need to show that the diagram
\[\xymatrix{
	A \ar[d]_{1_{A}} \ar[r]^-{\varphi}  & X \ar[d]^-{\omega}  \\
	A \ar[r]_-{\alpha}   & B }
\]
is commutative. It is easy to see that the above diagram is commutative since $(A,X,\varphi)$ is a lifting of $(A,B,\alpha)$, i.e. $\omega\varphi=\alpha$. Finally, let $x\in X$ and $a\in A$. Then
\[
1_A(x\cdot a) = x\cdot a = \omega(x)\cdot a = \omega(x)\cdot 1_A(a).
\]
Hence $\langle 1_A, \omega \rangle\colon (A,X,\varphi)\rightarrow(A,B,\alpha)$ is a morphism of generalized crossed module. This completes the proof.
\end{proof}

\begin{corollary}
Any lifting $(A,X,\varphi)$ of an aspherical generalized crossed module $(A,B,\alpha)$ is aspherical.
\end{corollary}

\begin{example}
Let $(H,G,i)$ be a generalized crossed module as in Example \ref{ex:subgrpgxmod}. Since $i$ is injective $(H,G,i)$ is aspherical and thus any lifting $(H,X,\varphi)$ of $(H,G,i)$ is also aspherical. Moreover $H$ can be considered as a subgroup of $X$ which is invariant under the self action of $X$.
\end{example}

\begin{lemma}
Let $(A,X,\varphi)$ be a lifting of a generalized crossed module $(A,B,\alpha)$ over $\omega\colon X\rightarrow B$ and let there exist isomorphisms $f\colon X\rightarrow X'$ and $g\colon B\rightarrow B'$ in \Gwa{}. Then $(A,X',\varphi')$ is a lifting of $(A,B',\alpha')$ over $\omega':=g\omega f^{-1}$ where $\varphi'=f\varphi$ and $\alpha'=g\alpha$. 
\end{lemma}

\begin{proof}
We know from Lemma \ref{defcond:genxmod1} that $(A,X',\varphi')$ and $(A,B',\alpha')$ are generalized crossed modules. So we only need to prove the commutativity, i.e. $\omega'\varphi'=\alpha'$. Since $(A,X,\varphi)$ be a lifting of a generalized crossed module $(A,B,\alpha)$ over $\omega\colon X\rightarrow B$ then $\omega\varphi=alpha$. So
\begin{alignat*}{2}
\omega'\varphi' & = (g\omega f^{-1})(f\varphi) \\
& = g\omega (f^{-1}f)\varphi \\
& = g\omega\varphi \\
& = g\alpha \\
& = \alpha'.
\end{alignat*}
This completes the proof.
\end{proof}

\begin{proposition}
Let $(A,X,\varphi)$ be a lifting of $(A,B,\alpha)$ over $\omega\colon X\rightarrow B$ and let $(A,X',\varphi')$ be a lifting of $(A,X,\varphi)$ over $\omega'\colon X'\rightarrow X$. Then $(A,X',\varphi')$ is a lifting of $(A,B,\alpha)$ over $\omega\omega'\colon X'\rightarrow B$.
\end{proposition}

\begin{proof}
It is sufficient to show the commutativity, i.e. $(\omega\omega')\varphi'=\alpha$. Since $(A,X,\varphi)$ is a lifting of $(A,B,\alpha)$ over $\omega\colon X\rightarrow B$ and $(A,X',\varphi')$ is a lifting of $(A,X,\varphi)$ over $\omega'\colon X'\rightarrow X$ then $\omega\varphi=\alpha$ and $\omega'\varphi'=\varphi$. So
\[(\omega\omega')\varphi'=\omega(\omega'\varphi')=\omega\varphi=\alpha.\]
This completes the proof.
\end{proof}

Let $(A,X,\varphi)$ and $(A,X',\varphi')$ be two liftings of a generalized crossed module $(A,B,\alpha)$ over $\omega\colon X\rightarrow B$ and $\omega'\colon X'\rightarrow B$, respectively. If there is a morphism $f\colon X\rightarrow X'$ such that $\omega'f=\omega$ then $f$ is called a morphism of liftings of $(A,B,\alpha)$.

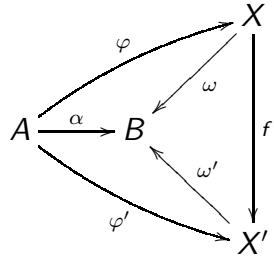
\[\xymatrix@R=1cm@C=1cm{
	 &  & X \ar[dd]^-{f} \ar[dl]^-{ \omega }  \\
A \ar[r]^-{\alpha} \ar@/^/[urr]^-{\varphi} \ar@/_/[drr]_-{\varphi'}	& B    &  \\
     &  & X' \ar[ul]_-{\omega'} }
\]

For a given generalized crossed module $(A,B,\alpha)$ the liftings of $(A,B,\alpha)$ form a category with the morphisms defined above. This category is denoted by $\mathsf{\bf LGXM}/(A,B,\alpha)$.

\begin{lemma}
Let $f\colon X\rightarrow X'$ be a morphism between liftings $(A,X,\varphi)$ and $(A,X',\varphi')$ of $(A,B,\alpha)$ over $\omega\colon X\rightarrow B$ and $\omega'\colon X'\rightarrow B$ respectively. Then $(A,X,\varphi)$ is a lifting of $(A,X',\varphi')$ over $f$ if $\omega'$ is a monomorphism.
\end{lemma}

\begin{proof}
Let $\omega'\colon X'\rightarrow B$ be a monomorphism.
Since $(A,X,\varphi)$ and $(A,X',\varphi')$ are liftings of $(A,B,\alpha)$ over $\omega\colon X\rightarrow B$ and $\omega'\colon X'\rightarrow B$, respectively, then $\omega\varphi=\alpha=\omega'\varphi'$. Hence
\[\omega'(f\varphi)=(\omega'f)\varphi=\omega\varphi=\omega'\varphi'\] and since $\omega'\colon X'\rightarrow B$ is a monomorphism then $f\varphi=\varphi'$. Thus $(A,X,\varphi)$ is a lifting of $(A,X',\varphi')$ over $f$.
\end{proof}

\begin{proposition}
Let $\langle f,g\rangle\colon (\widetilde{A},\widetilde{B},\widetilde{\alpha})\rightarrow(A,B,\alpha)$ be a morphism of generalized crossed modules and let $(\widetilde{A},\widetilde{X},\widetilde{\varphi})$ be a lifting of $(\widetilde{A},\widetilde{B},\widetilde{\alpha})$ over $\widetilde{\omega}\colon \widetilde{X}\rightarrow \widetilde{B}$. Then $\langle f,\widetilde{g}\rangle\colon (\widetilde{A},\widetilde{X},\widetilde{\varphi})\rightarrow(A,B,\alpha)$ is a morphism of generalized crossed modules where $\widetilde{g}=g\widetilde{\omega}$.	
\end{proposition}

\begin{proof}
It is sufficient to prove the commutativity of the following diagram.
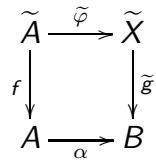
\[\xymatrix{
	\widetilde{A} \ar[d]_{f} \ar[r]^-{\widetilde{\varphi}}  & \widetilde{X} \ar[d]^-{\widetilde{g}}  \\
	A \ar[r]_-{\alpha}   & B }
\]
Since $\widetilde{g}=g\widetilde{\omega}$ then
\[\widetilde{g}\widetilde{\varphi}=(g\widetilde{\omega})\widetilde{\varphi}=g(\widetilde{\omega}\widetilde{\varphi})=g\widetilde{\alpha}=\alpha f\]
and this completes the proof.
\end{proof}

Following theorem gives a criteria for lifting a generalized crossed module morphism to any lifting.

\begin{theorem}
Let $\langle f,g\rangle\colon (\widetilde{A},\widetilde{B},\widetilde{\alpha})\rightarrow(A,B,\alpha)$ be a morphism of generalized crossed modules such that $(\widetilde{A},\widetilde{B},\widetilde{\alpha})$ is simply connected and let $(A,X,\varphi)$ be a lifting of $(A,B,\alpha)$ over $\omega\colon X\rightarrow B$. Then there exist a generalized crossed module morphism $\langle f,\widetilde{g}\rangle\colon (\widetilde{A},\widetilde{B},\widetilde{\alpha})\rightarrow(A,X,\varphi)$ such that $\omega\widetilde{g}=g$ if and only if $f(\ker \widetilde{\alpha})\subseteq \ker \varphi$.
\end{theorem}

\begin{proof}
First assume that $f(\ker \widetilde{\alpha})\subseteq \ker \varphi$. Let $\widetilde{g}\colon\widetilde{B}\rightarrow X$ is given by $\widetilde{g}(\widetilde{b})=\varphi f(\widetilde{a})$ for some $\widetilde{a}\in \widetilde{A}$ such that $\widetilde{\alpha}(\widetilde{a})=\widetilde{b}$. The function $\widetilde{g}$ is well-defined since $(\widetilde{A},\widetilde{B},\widetilde{\alpha})$ is simply connected, i.e. $\widetilde{\alpha}$ is surjective, and $f(\ker \widetilde{\alpha})\subseteq \ker \varphi$. It is easy to see that $\widetilde{g}$ is a morphism in \Gwa{}. Other details are straightforward so omitted.

Conversely, let $\widetilde{x}\in\ker\widetilde{\alpha}$. Then by the definition of $\widetilde{g}$
\[\varphi(f(\widetilde{x}))=\widetilde{g}(\widetilde{\alpha}(\widetilde{x}))=\widetilde{g}(0)=0\]
and thus $f(\widetilde{x})\in\ker \varphi$. This completes the proof.
\end{proof}

\begin{corollary}
For a generalized crossed module $(A,B,\alpha)$ with two liftings $(A,X,\varphi)$ and $(A,\widetilde{X},\widetilde{\varphi})$ via $\omega$ and $\widetilde{\omega}$, respectively, such that $(A,\widetilde{X},\widetilde{\varphi})$ is simply connected. In this case $\ker\widetilde{\varphi}\subseteq\ker\varphi$ if and only if $(A,\widetilde{X},\widetilde{\varphi})$ is also a lifting of $(A,X,\varphi)$.
\end{corollary}

\begin{corollary}
For a generalized crossed module $(A,B,\alpha)$ with two simply connected liftings $(A,X,\varphi)$ and $(A,\widetilde{X},\widetilde{\varphi})$ via $\omega$ and $\widetilde{\omega}$, respectively.Then $\ker\widetilde{\varphi}=\ker\varphi$ if and only if there is an isomorphism $\langle f,g\rangle\colon(A,X,\varphi)\rightarrow(A,\widetilde{X},\widetilde{\varphi})$ of generalized crossed modules.
\end{corollary}

\begin{theorem}
Let $(A,B,\alpha)$ be a generalized crossed module with $N$ ideal of $A$ such that $N\subseteq\ker\alpha$. Then $(A,A/N,\varphi)$ is a lifting of $(A,B,\alpha)$ via $\omega$ where $\varphi(a)=a+N$ and $\omega(a+N)=\alpha(a)$. Moreover, $\ker\varphi=N$. 
\end{theorem}

\begin{proof}
It is easy to see that $\omega\varphi=\alpha$. Thus, it is sufficient to show that the conditions for Lemma \ref{lem:liftprexmod} are satisfied. 
Then for any $a,a_1\in A$ we have
\begin{alignat*}{2}
	{\varphi((a+N)\cdot a_1)} & = {\varphi(\alpha(a)\cdot a_1)} \\
	& = {\varphi({}^{a}{a_1})} \\
	& = {}^{\varphi(a)}{\varphi(a_1)} \\
	& = {}^{a+N}{\varphi(a_1)}
\end{alignat*}
which completes the proof.
\end{proof}

\begin{theorem}
Let $(A,B,\alpha)$ be a generalized crossed module. Then there is a natural equivalence  between the category $\mathsf{\bf Cov_{GXMod}}(A,B,\alpha)$ of covering generalized crossed modules of $(A,B,\alpha)$ and the category $\mathsf{\bf LGXM}/(A,B,\alpha)$ of lifting generalized crossed modules of $(A,B,\alpha)$.
\end{theorem}

\begin{proof}
It is easy to see that if $(A,X,\varphi)$ is a lifting of $(A,B,\alpha)$ over $\omega\colon X\rightarrow B$, then $\langle1_A,\omega\rangle\colon(A,X,\varphi)\rightarrow (A,B,\alpha)$ is a generalized crossed module morphism. Since $1_A$ is an isomorphism then $(A,X,\varphi)$ is a covering of $(A,B,\alpha)$. Moreover, let $(\wt{A},\wt{X},\wt{\varphi})$ be another lifting of $(A,B,\alpha)$ over $\wt{\omega}\colon\wt{X}\rightarrow B$ and let $f\colon X\rightarrow \wt{X}$ be a morphism in $\mathsf{\bf LGXM}/(A,B,\alpha)$, then $\langle1_A,f\rangle\colon(A,X,\varphi)\rightarrow(A,\widetilde{X},\widetilde{\varphi})$ is a morphism in $\mathsf{\bf Cov_{GXMod}}(A,B,\alpha)$. This construction defines a functor $\mathsf{\bf LGXM}/(A,B,\alpha)\rightarrow\mathsf{\bf Cov_{GXMod}}(A,B,\alpha)$.
	
Conversely if $\langle f,g\rangle\colon(\wt{A},\wt{B},\wt{\alpha})\rightarrow (A,B,\alpha)$ is a covering morphism of generalized crossed modules, then $f\colon \wt{A}\rightarrow A$ is an isomorphism and if we take $\varphi=\wt{\alpha}f^{-1}$ then $(A,\wt{B},\wt{\varphi})$ becomes a lifting of $(A,B,\alpha)$ over $g$. It is easy to see that $g\varphi=\alpha$ and $(A,\widetilde{B},\varphi)$ is a generalized crossed module. Further  if $\langle f',g'\rangle\colon(A',B',\alpha')\rightarrow (A,B,\alpha)$ is another covering morphism and $\langle\wt{f},\wt{g}\rangle\colon(\wt{A},\wt{B},\wt{\alpha})\rightarrow (A',B',\alpha')$ is a morphism in $\mathsf{\bf Cov_{GXMod}}(A,B,\alpha)$ then $\wt{g}\colon\widetilde{B}\rightarrow  B'$ is a morphism in $\mathsf{\bf LGXM}/(A,B,\alpha)$. This construction defines a functor $\mathsf{\bf Cov_{GXMod}}(A,B,\alpha)\rightarrow\mathsf{\bf LGXM}/(A,B,\alpha)$. Other details can be proven by easy calculations.
\end{proof}


\begin{thebibliography}{10}
	
	\bibitem{Akiz2013}
	Ak{\i}z, H.~F., Alemdar, N., Mucuk, O. and {\c{S}}ahan, T., Coverings of
	internal groupoids and crossed modules in the category of groups with
	operations, \emph{Georgian Math. J.} \textbf{20(2)}, (2013), 223--238.
	
	\bibitem{Brown2011}
	Brown, R., Higgins, P.~J. and Sivera, R., \emph{Nonabelian Algebraic Topology:
		Filtered Spaces, Crossed Complexes, Cubical Homotopy Groupoids}, EMS series
	of lectures in mathematics, European Mathematical Society (2011).
	
	\bibitem{Brown1982a}
	Brown, R. and Huebschmann, J., Identities among relations, in:
	\emph{Low-Dimensional Topology}, Cambridge University Press (1982), 153--202.
	
	\bibitem{Brown1994}
	Brown, R. and Mucuk, O., Covering groups of non-connected topological groups
	revisited, \emph{Math. Proc. Camb. Phil. Soc.} \textbf{115(1)}, (1994),
	97--110.
	
	\bibitem{Brown1976}
	Brown, R. and Spencer, C.~B., G-groupoids, crossed modules and the fundamental
	groupoid of a topological group, \emph{Indagat. Math.} \textbf{79(4)},
	(1976), 296--302.
	
	\bibitem{Datuashvili2002a}
	Datuashvili, T., Central series for groups with action and leibniz algebras,
	\emph{Georgian Math. J.} \textbf{9(4)}, (2002), 671--682.
	
	\bibitem{Datuashvili2004}
	Datuashvili, T., Witt's theorem for groups with action and free leibniz
	algebras, \emph{Georgian Math. J.} \textbf{11(4)}, (2004), 691--712.
	
	\bibitem{Huebschmann1980}
	Huebschmann, J., Crossed $n$-folds extensions of groups and cohomology.,
	\emph{Comment. Math. Helv.} \textbf{55}, (1980), 302--313.
	
	\bibitem{Demir2021}
	Karaka\c{s}, S.~D. and Mucuk, O., Liftings and covering morphisms of crossed
	modules in group-groupoids, \emph{Turk. J. Math.} \textbf{45(3)}, (2021),
	1407--1417.
	
	\bibitem{Loday1978}
	Loday, J.-L., Cohomologie et groupe de {S}teinberg relatifs, \emph{J. Algebra}
	\textbf{54(1)}, (1978), 178--202.
	
	\bibitem{Loday1982}
	Loday, J.-L., Spaces with finitely many non-trivial homotopy groups, \emph{J.
		Pure. Appl. Algebra} \textbf{24(2)}, (1982), 179--202.
	
	\bibitem{Lue1979}
	Lue, A. S.-T., Semi-complete crossed modules and holomorphs of groups, \emph{B.
		Lond. Math. Soc.} \textbf{11(1)}, (1979), 8--16.
	
	\bibitem{Mucuk2020}
	Mucuk, O. and Ak{\i}z, H.~F., Covering morphisms of topological internal
	groupoids, \emph{Hacet. J. Math. Stat.} \textbf{49}, (2020), 1020--1029.
	
	\bibitem{Mucuk2014}
	Mucuk, O. and \c{S}ahan, T., Coverings and crossed modules of topological
	groups with operations, \emph{Turk. J. Math.} \textbf{38(5)}, (2014),
	833--845.
	
	\bibitem{Mucuk2019}
	Mucuk, O. and {\c{S}}ahan, T., Group-groupoid actions and liftings of crossed
	modules, \emph{Georgian Math. J.} \textbf{26(3)}, (2019), 437--447.
	
	\bibitem{Mucuk2015}
	Mucuk, O., {\c{S}}ahan, T. and Alemdar, N., Normality and quotients in crossed
	modules and group-groupoids, \emph{Appl. Categor. Struct.} \textbf{23(3)},
	(2015), 415--428.
	
	\bibitem{Orzech1972}
	Orzech, G., Obstruction theory in algebraic categories, {I}, \emph{J. Pure.
		Appl. Algebra} \textbf{2(4)}, (1972), 287--314.
	
	\bibitem{Orzech1972a}
	Orzech, G., Obstruction theory in algebraic categories, {II}, \emph{J. Pure.
		Appl. Algebra} \textbf{2(4)}, (1972), 315--340.
	
	\bibitem{Porter1987}
	Porter, T., Extensions, crossed modules and internal categories in categories
	of groups with operations, \emph{P. Edinburgh Math. Soc.} \textbf{30(3)},
	(1987), 373--381.
	
	\bibitem{Tuncar2019}
	{\c{S}}ahan, T., Further remarks on liftings of crossed modules, \emph{Hacet.
		J. Math. Stat.} \textbf{48(3)}, (2019), 743--752.
	
	\bibitem{Temel2020a}
	Temel, S., Crossed squares, crossed modules over groupoids and cat$^{\bf
		{1-2}}-$groupoids, \emph{Categ. Gen. Algebr. Struct. Appl.} \textbf{13(1)}, (2020), 125--142.
	
	\bibitem{Temel2020}
	Temel, S., {\c{S}}ahan, T. and Mucuk, O., Crossed modules, double
	group-groupoids and crossed squares, \emph{Filomat} \textbf{34(6)}, (2020), 1755--1769.
	
	\bibitem{Whitehead1946}
	Whitehead, J. H.~C., Note on a previous paper entitled "{O}n adding relations
	to homotopy groups", \emph{Ann. Math.} \textbf{47(4)}, (1946), 806--810.
	
	\bibitem{Whitehead1949}
	Whitehead, J. H.~C., Combinatorial homotopy. {II}, \emph{Bull. Amer. Math.
		Soc.} \textbf{55(5)}, (1949), 453--496.
	
	\bibitem{Yavari2019}
	Yavari, M. and Salemkar, A., The category of generalized crossed modules,
	\emph{Categ. Gen. Algebr. Struct. Appl.} \textbf{10(1)}, (2019), 157--171.
\end{thebibliography}
\end{document}